\documentclass[]{amsart}

\def\bn{\mathbb{N}}
\def\bz{\mathbb{Z}}
\def\f{\mathbb{F}}

\newcommand{\com}[1]{#1^{\prime}}
\newcommand{\bcom}[1]{#1^{\prime\prime}}
\newcommand{\inner}[2]{\langle #1,#2\rangle}

\newcommand{\lv}{{\rm L}(V)}
\newcommand{\lvv}{{\rm L}(V^*)}

\newcommand{\luv}{{\rm L}(U,V)}
\newcommand{\eu}{{\rm L}_R(U)}
\newcommand{\ev}{{\rm L}_R(V)}
\newcommand{\ew}{{\rm L}_R(W)}
\newcommand{\evv}{{\rm L}(V^*)_R}

\newcommand{\tr}{\tilde{R}}
\newcommand{\fv}{{\rm F}(V)}

\newcommand{\ann}[2]{{\rm ann}_{#1}(#2)}
\newcommand{\im}[1]{{\rm im}\,{#1}}
\newcommand{\huv}{{\rm L}_R(U,V)}
\newcommand{\hvu}{{\rm L}(V^*,U^*)_R}

\newcommand{\du}[1]{#1^{\sharp}}

\newtheorem{theorem}{Theorem}[section]
\newtheorem{lemma}[theorem]{Lemma}
\newtheorem{co}[theorem]{Corollary}
\newtheorem{pr}[theorem]{Proposition}
\theoremstyle{remark}

\theoremstyle{definition}
\newtheorem{de}[theorem]{Definition}
\newtheorem{ex}[theorem]{Example}
\numberwithin{equation}{section}

\begin{document}

\title[]{Bicommutants and ranges of derivations} 
\author{Bojan Magajna} 
\address{Department of Mathematics\\ University of Ljubljana\\
Jadranska 21\\ Ljubljana 1000\\ Slovenia}
\email{Bojan.Magajna@fmf.uni-lj.si}

\thanks{Acknowledgment. I am grateful to Matej Bre\v sar for his comment of the first
draft of the paper.} 

\thanks{The author was supported in part by the Ministry of Science
and Education of Slovenia.}



\begin{abstract}Let $V$ be a vector space over a field $\f$, $V^*$ its dual space and $\lv$ the algebra of all linear operators on $V$.
For an operator $a\in\lv$ let $a^*$ be its adjoint acting on $V^*$, and for a subset $R$ of $\lv$ let $\bcom{R}$ be its 
bicommutant. 
If $R$ is the subalgebra of $\lv$ generated by an operator $a$, we prove that  
$Z:=\{b^*:\, b\in R\}^{\prime\prime}\subseteq\{b^*:\, b\in\bcom{R}\}$; moreover $Z$ is described. This inclusion
is equality if $V$ as a module over
the polynomial algebra $\f[t]$ via $t\mapsto a$ is nice enough (say torsion, or injective, or if it contains a copy of 
$R$ as a direct summand).  Further, under the same assumption about $V$ for any $b\in\lv$, $b\in\bcom{(a)}$ if and only if 
the  derivations $d_a$ and $d_b$ satisfy $d_b(\fv)\subseteq d_a(\fv)$, where $\fv$ is the
set of all finite rank operators on $V$. The inclusion 
$d_b(\lv)\subseteq d_a(\lv)$ also holds under these conditions. 
\end{abstract}

\maketitle

\section{Introduction}
For a subset $R$ of the algebra $\lv$ of all linear operators on a vector space $V$ over a field $\f$ let $\com{R}$ be its commutant (= the set of all operators in $\lv$ that commute with
all elements of $R$) and $\bcom{R}=(\com{R})^{\prime}$ its bicommutant. We denote by
$\com{(a)}$ and $\bcom{(a)}$ the commutant and the bicommutant of a single operator
$a\in\lv$. As usual, $a^*$ denotes the adjoint of $a$, acting on the dual space $V^*$.

If $b\in\lv$ is such that
$b^*\in\bcom{(a^*)}$, then $b^*$ commutes with all operators $e\in\com{(a^*)}$, hence in particular with all operators
of the form $e=c^*$, where $c\in\com{(a)}$. Then  $bc=cb$, hence $b\in\bcom{(a)}$. This proves that
$b^*\in\bcom{(a^*)}$ implies that $b\in\bcom{(a)}$. Is the reverse of this implication also true? That is: 

\medskip
{\em Does $b\in\bcom{(a)}$ imply that $b^*\in\bcom{(a^*)}$?}

\medskip
To be in $\bcom{(a^*)}$ the operator $b^*$ must commute with each $e\in\com{(a^*)}$, but since not all such  $e$
are adjoints of operators on $V$, there is no obvious reason for $b^*$ to be in $\bcom{(a^*)}$. Indeed, the answer to
the analogous question in the context of bounded operators on Banach spaces is negative (Section 5).
But it is perhaps  surprising that in many cases the answer for general linear operators is positive.

Observe that for a subalgebra $R\subseteq\lv$, $\com{R}$ is just the algebra $\ev$ of all
$R$-module endomorphisms of $V$, and that $V^*$ is a right $R$-module by $\inner{\rho r}{\xi}:=\inner{\rho}{r\xi}$,
where $\rho\in V^*$, $\xi\in V$ and $r\in R$. (Here we are using the convenient notation $\inner{\rho}{\xi}$
for the value of a functional $\rho$ at the vector $\xi$.) 

In Section 2, we will
show  that for any $R$ modules $U$ and $V$ the right $R$-module homomorphisms  $g\in{\rm L}(V^*,U^*)_R$ can be 
interpolated on 
finite subsets of $V^*$ and $U$ by adjoints of maps  $f\in{\rm L}_R(U,V)$ (Theorem \ref{th1}), if $V$ as an $R$ module is
injective or if $U$ as an $R$ module is such that each finite subset of $U$ is contained in a complemented {\em finitely related}
submodule of $U$. Here an $R$-module $W$ is called {\em finitely related} if $W$ is isomorphic to a quotient of a free module by
a finitely generated submodule. (For example, each finitely generated module over a left noetherian algebra is finitely related
\cite{La}.) We recall that an A-module $V$ is called {\em injective}
if it is a direct summand in any module in which it is a submodule \cite[p. 204]{AF}, \cite[p. 60]{La}. This is equivalent to the
requirement that for each pair of modules $U\subseteq W$ every module homomorphism $f:U\to V$ extends to a module
homomorphism $W\to V$. We will be concerned mainly with modules over a principal ideal domain $R$, so we recall
that such a module $V$ is injective if and only if it is {\em divisible} (see \cite{Ro} or \cite[p. 71]{La}), which means
that for every $\eta\in V$ and every nonzero $p\in R$ there exists $\xi\in V$ such that $p\xi=\eta$.

For any operator $a\in\lv$ we study in Section 3 the bicommutant of $a^*$. We regard $V$ as a module over the polynomial
algebra $R=\f[t]$ via the correspondence $t\mapsto a$, so that $V^*$ is also an $R$-module via $t\mapsto a^*$.
Then $Z:=\bcom{(a^*)}$ is just the center of the endomorphism algebra $\evv$. It turns out that if $V$ is torsion-free, then 
$Z$ is isomorphic to a subalgebra of the algebra $K:=\f(t)$ of rational functions. (Recall that an $R$-module $V$ is {\em torsion} 
if each $\xi\in V$ is annihilated by some nonzero $p\in R$, 
that is $p\xi=0$. Thus  $V$ is torsion means that $a$ is  
{\em locally algebraic}. On the other hand, $V$ is called {\em 
torsion-free} if $p\xi\ne0$ for all nonzero $p\in R$ and $\xi\in V$.)
We will show that if the torsion submodule of $V$
is not $0$, but $V$ is not torsion, then each element $b$ of $Z$ induces a decomposition $V=T_b\oplus W_b$,
such that $b$ acts on $W_b$ as the multiplication by a rational function of $a$, and $T_b$ is a finite direct sum of
torsion submodules of bounded torsion, on each of which $b$ acts as a polynomial in $a$. As a consequence it will follow
that $Z$ is contained in the set $\{c^*:\, c\in Z_0\}$ where $Z_0=\bcom{(a)}$ ($Z_0$ is just the center of $\ev$).  The inclusion $Z\subseteq Z_0$ turns out
to be an equality in many cases, for example, if $V$ as an $R$-module is torsion, or injective or contains a copy of $R$
as a direct summand.

Our  initial motivation for studying the above question was the range inclusion problem for derivation ranges. If
$a$ is an element of an algebra $A$, the {\em  derivation}  induced by $a$ is the map $d_a$ on $A$, defined by
$d_a(x)=ax-xa$. The kernel of $d_a$ is just the commutant of $a$ in $A$, but the range of $d_a$ also turns out to be 
interesting. If $b$ is another element of $A$, we may ask, when is the range of $d_b$ contained in the range of
$d_a$. This problem was studied in the past by several authors, especially in the case when $A$ is the algebra of all
bounded operators on a Hilbert space $H$. Very interesting results were obtained by Johnson and Williams \cite{JW} in 
the case $a$ is a normal operator, and their work was continued  for example by Fong \cite{F}, 
Kissin and Shulman \cite{KS}, Bre\v sar \cite{Br} and in \cite{BMS}. 
Some of their results  are of such a nature that one would expect them to hold for much larger class
of operators $a$ than just normal ones. But when trying to show this in a complete generality we encountered certain 
analytic difficulties. We found, however, that the problem is interesting also in the purely algebraic 
context and, since the methods in this case are completely different from those required for bounded operators,
we decided to study this case separately. With $a\in\lv$ such that $Z=Z_0$ (where $Z$ and $Z_0$ are as in the previous 
paragraph), we will 
show in Section 4 that for any $b\in\lv$ the condition $b\in\bcom{(a)}$ is 
equivalent to $d_b(\fv)\subseteq d_a(\fv)$, where $\fv$ is the ideal in $\lv$ of finite rank operators. Then also 
the inclusions $d_{b^*}(\lvv)\subseteq d_{a^*}(\lvv)$ and  $d_b(\lv)\subseteq d_a(\lv)$ hold.

\section{Bicommutants and adjoints}

\subsection{Approximation of operators on $V^*$ by adjoints of operators on $V$} For vector spaces $U$ and $V$ over a field
$\f$ we denote by $\luv$ the space of all linear operators from $U$ to $V$. As usual, regard any vector space $V$ as a 
subspace in its bidual $V^{**}$ through the natural map $V\to V^{**}$.
\begin{lemma}\label{le1}For each $a\in \luv$ every element $\theta\in\ker{a^{**}}$ can be approximated by elements from
$\ker{a}$ in the following sense: for each finite
subset $\{\rho_j: j=1,\ldots,n\}$ of $U^*$ there exist $\xi\in \ker{a}$ such that $\inner{\rho_j}{\xi}=\inner{\rho_j}{\theta}$
for all $j$. 
\end{lemma}

\begin{proof}It is well-known (and elementary) that for every $a\in\luv$  the equality
\begin{equation}\label{1}\ker{a^*}=(\im{a})^{\perp}\end{equation}
holds, where $(\im a)^{\perp}$ is the annihilator of $\im a$ in $V^*$. Similarly   
\begin{equation}\label{2}\im{a^*}=(\ker{a})^{\perp}.\end{equation} 
(For a proof of the nontrivial inclusion $(\ker{a})^{\perp}\subseteq\im{a^*}$,  note that for each $\rho\in(\ker{a})^{\perp}$
the map $a\xi\mapsto\rho(\xi)$ is well defined on $\im{a}$, and any of its linear extensions $\omega\in V^*$ satisfies
$a^*(\omega)=\omega\circ a=\rho$.) Thus
$$\ker{a^{**}}=(\im{a^*})^{\perp}=(\ker{a})^{\perp\perp}.$$
Since for each subspace $W$ of $U$, $W^{\perp\perp}$ is naturally isomorphic to the bidual $W^{**}$ of $W$, we infer that
$\ker{a^{**}}=(\ker{a})^{**}$ (where `=' means the natural isomorphism), so the lemma reduces to the well-known density 
of $W$ in $W^{**}$.
\end{proof}

\begin{theorem}\label{th1}Let $R$ be an $\f$-algebra (where $\f$ is a field) and $U$, $V$ any left
$R$-modules. If $U$ is finitely related then each $g\in\hvu$ can be approximated by adjoints of elements of $\huv$ in the following sense:
for every finite subsets $G$ of $U$ and $H$ of $V^*$ there exists $f\in\huv$ such that
$$\inner{g(\rho)}{\xi}=\inner{\rho}{f(\xi)}\ \ \mbox{for all}\ \rho\in H\ \mbox{and}\ \xi\in G.$$
The same conclusion holds also if $V$ is injective (for a general $U$) or if $U$ is such that each finite
subset of $U$ is contained in a complemented finitely presented submodule of $U$.
\end{theorem}

\begin{proof}Let us first consider the case when $U$ is finitely related, hence of the
form $U=R^{(J)}/A$ for some index set $J$ and a finitely generated left submodule $A$ of $R^{(J)}$. (Here $R^{(J)}$ denotes a free module,
the submodule of the cartesian power $R^{J}$, consisting of elements which have only finitely many nonzero components.)
Let $\{r_1,\ldots,r_m\}$ be a set of generators of $A$. Since the space ${\rm L}_R(R^{(J)},V)$
can be naturally identified with $V^J$, we thus have a natural isomorphism
$${\rm L}_R(U,V)=\{f\in{\rm L}_R(R^{(J)},V):\, f(A)=0\}=\ann{V^J}{A}.$$
Under this isomorphism a map $f\in\huv$ corresponds to the element $$(fq(e_j)))\in\ann{V^J}{A},$$ where
the $e_j$ are the usual basic elements of $R^{(J)}$ ($e_j$ has $1$ on the $j$-th position and $0$ elsewhere) and
$q:R^{(J)}\to U$ is the quotient map.
Similarly, using the natural isomorphism
$${\rm L}(V^*,U^*)_R={\rm L}_R(U,V^{**})\ \ (g\mapsto g^*|U),$$
we have that
$${\rm L}(V^*,U^*)_R=\ann{(V^{**})^J}{A}.$$
The space $(V^{**})^J$ is the dual of $(V^*)^{(J)}$ and it can be verified that under the above identifications the
theorem translates to the following statement: given $\theta=(\theta_j)\in(V^{**})^J$, for each finite subset $H_0$ of 
$(V^*)^{(J)}$ there exists $v=(v_j)\in V^J$ such that $\inner{\rho}{v}=\inner{\rho}{\theta}$ for all $\rho\in H_0$. (Here the fact that
$(V^*)^{(J)}$ is an $R$-module is used.)
Denoting by $r_{i,j}$ the components of the generators $r_i$ of $A$, an element $v=(v_j)\in V^J$ (respectively
an element $\theta=(\theta_j)\in (V^{**})^J$) is in
$\ann{V^J}{A}$ (resp. in $\ann{(V^{**})^J}{A}$) if and only if
$$\sum_{j\in J}r_{i,j}v_j=0\ \ (\mbox{resp.} \sum_{j\in J}r_{i,j}\theta_j=0)\ \ \mbox{for all}\ i=1,\ldots m.$$
The sums here are finite since each $r_i$ has only finitely many non-zero components. Let $n$ be a finite subset
of $J$ such that $r_{i,j}=0$ for all $i\in\{1,\ldots,m\}$ if $j\in J\setminus n$. Let $r:V^n\to V^m$ be the map
defined by $(r((v_j)))_i=\sum_{j\in n}r_{i,j}v_j$. Then $r^{**}:(V^{**})^n\to (V^{**})^m$ is given by
$(r^{**}((\theta_j)))_i=\sum_{j\in n}r_{i,j}^{**}\theta_j$ and it follows that
$$\ann{V^J}{A}=\ker r\times V^{J\setminus n}\ \ \mbox{and}\ \ 
\ann{(V^{**})^J}{A}=\ker r^{**}\times(V^{**})^{J\setminus n}.$$
For finitely presented modules the theorem follows now from Lemma \ref{le1} (and the density of a space in its bidual)
and the rest of the theorem is an immediate consequence. In fact, it is sufficient to require that each finite
subset $F$ of $U$ is contained in a finitely related submodule $U_F$ of $U$ such that each homomorphism
$U_F\to V$ can be extended to a homomorphism from $U$ to $V$.
\end{proof}

The condition in the last sentence of the proof of Theorem \ref{th1} is satisfied, for example in the case
when $R$ is left noetherian, $U=V$ and each finite subset $F$ of $V$ is contained in a quasiinjective 
(in the sense of \cite[6.74]{La}) 
submodule of $V$, but we will not use this in the paper.
Now we show by an example that  Theorem \ref{th1} can not be extended to general modules.
\begin{ex}Let $V$ be an infinite dimensional vector space and $W$ a weak* dense subspace of $V^*$, different
from $V^*$. (For example, $W=\ker\theta$ for some $\theta\in V^{**}\setminus V$.) Let $R=\lv$,  $J=\{a\in R:\,
\im{a^*}\subseteq W\}$ (so that $J$ is a left ideal in $R$) and $U:=R/J$. Then, denoting by $S_{\perp}$ and $S^{\perp}$
the annihilators of a subset $S\subseteq V^*$ in $V$ and in $V^{**}$ (respectively), and  identifying ${\rm L}_R(R,V)$ 
with $V$, we have that
$$\huv=\ann{V}{J}=\cap_{a\in J}\ker{a}=(\sum_{a\in J}\im{a^*})_{\perp}=W_{\perp}=0,$$
but
$$\hvu={\rm L}_R(U,V^{**})=\cap_{a\in J}\ker{a^{**}}=(\sum_{a\in J}\im{a^*})^{\perp}=W^{\perp}\ne0.$$
\end{ex}
The above example leaves open the possibility that Theorem \ref{th1} might be true for general modules over
nice algebras $R$, such as $R=\f[t]$. To see that this is not the case, let $V=R$ and let $U$ be the field $K=\f(t)$ of
all rational functions over $\f$ (that is, the  field of quotients of $R$). Then $\huv=0$ since $U$ is divisible 
and $V$ has no nonzero divisible $R$-submodules. But $\hvu\ne0$ since $U^*$ is divisible (namely, a vector space over $K$),
hence injective, while  as we show in the example below, $V^*$ contains a copy $R_0$ of $R$. (Take any nonzero
homomorphism from $R_0$ to $U^*$ and extend it to a homomorphism from $V^*$ to $U^*$.)

\begin{ex}\label{ex1}Let $V=\f^{(\bn)}$ be the vector space of all sequences with the entries in a field $\f$ 
that have only finitely many nonzero terms. Let $a\in\lv$ be the
shift to the right:
$$a(\xi_0,\xi_1,\ldots)=(0,\xi_0,\xi_1,\ldots).$$
$V$ has a cyclic vector $\xi=(1,0,0,\ldots)$ for $a$  and $p(a)\xi\ne0$ for all nonzero polynomials $p\in R$, 
hence $V$, as a module over $R$, is isomorphic
to $R$. Since $R$ is commutative this implies that $\bcom{(a)}=\com{(a)}\cong{\rm L}_R(R)=R$. 

The dual $R^*$ of $R$ is isomorphic to $V^*$, hence to the module $\f^{\bn}$ of all sequences with the entries in $\f$,
where the module operation is given by the backward shift $a^*(\xi_0,\xi_1,\ldots)=(\xi_1,\xi_2,\ldots)$.
Observe that $V^*$ is not a torsion module: it is possible to recursively define $\eta_n\in\f$ such that all the 
translates $(a^*)^n\eta$ ($n\in\bn$) of the vector $\eta:=(\eta_0,\eta_1,\ldots)$ are linearly independent, hence
$p(a^*)\eta\ne0$ for all nonzero polynomials $p\in\f[t]$. Indeed, set $\eta_0=1$ and assume inductively that
$\eta_0,\ldots,\eta_{2n}$ have already been found such that the determinant 
$$\delta_n=\left|\begin{array}{llll}
\eta_0&\eta_1&\ldots&\eta_n\\
\eta_1&\eta_2&\ldots&\eta_{n+1}\\
\ldots&\ldots&\ldots&\ldots\\
\eta_n&\eta_{n+1}&\ldots&\eta_{2n}
\end{array}\right|$$
is nonzero. Then we can find $\eta_{2n+1},\eta_{2n+2}$ in $\f$ such that the corresponding determinant $\delta_{n+1}$
is nonzero. (We can even choose $\eta_{2n+1}=0$.)
\end{ex}

\subsection{Bicommutants and adjoints}
Theorem \ref{th1} has the following simple consequence.
\begin{co}\label{co2}If $V$ is injective over $R$ or if each finite subset of $V$ is contained in a 
finitely related complemented submodule of $V$ then 
$$\bcom{(\tr)}\supseteq\widetilde{\bcom{R}},\ \ \mbox{where}\ \tilde{R}:=\{a^*:\, a\in R\}\ \mbox{and}\ \widetilde{\bcom{R}}:=\{b^*:\, b\in\bcom{R}\}.$$
\end{co}

\begin{proof}For each $b\in\bcom{R}$ we have to show that  $b^*g=gb^*$ for all $g\in\com{\tr}$, or equivalently, that
$$\inner{b^*g\rho}{\xi}=\inner{gb^*\rho}{\xi}$$
for all $\xi\in V$ and $\rho\in V^*$. Given $g$, $\xi$ and $\rho$, by Theorem \ref{th1} there exists $f\in\com{R}$ such that
$$\inner{g\rho}{b\xi}=\inner{\rho}{fb\xi}\ \ \mbox{and}\ \ \inner{gb^*\rho}{\xi}=\inner{b^*\rho}{f\xi}.$$
Hence
\begin{align*}\inner{b^*g\rho}{\xi}=\inner{g\rho}{b\xi}=\inner{\rho}{fb\xi}=\inner{\rho}{bf\xi}
=\inner{b^*\rho}{f\xi}=\inner{gb^*\rho}{\xi}.
\end{align*}
\end{proof}

Observe that if $R$ is generated by a single operator $a$ (and the identity) Corollary \ref{co2} says that
the center $Z$ of the algebra $\evv$ contains all operators $b^*$ such that $b$ is in the center $Z_0$ of $\ev$.
At first the author was convinced that these two sets always coincide, but the following example shows that
this is not true. We will need the simple well-known fact that a module $V$ over $R=\f(t)$ is torsion-free
if and only if $V^*$ is divisible. 

\begin{ex}\label{ex11}Let $U$ and $W$ be torsion-free modules over $R=\f[t]$ such that there are no nonzero homomorphisms
from $U$ to $W$ and also from $W$ to $U$. For example, take two distinct primes $p,q\in R$ (such that $q$ is not
a constant multiple of $p$) and let $U$ and $W$ be the submodules of $K=\f(t)$ consisting of all rational functions
with the denumerators a power of $p$ and $q$, respectively. Let $V=U\oplus W$. Then $\ev$ is the direct sum $\eu\oplus\ew$,
hence the center $Z_0$ of $\ev$ is the direct sum of the centers $Z_1$ and $Z_2$ of $\eu$ and $\ew$. On the other hand,
since $U^*$ and $V^*$ are injective and non-torsion (for both have $R^*$ as a quotient, which is not torsion
by Example \ref{ex1}), there exist nonzero homomorphisms from $U^*$ to $V^*$, and such homomorphisms present an obstruction
for an element $z_1^*\oplus z_2^*$ ($z_i\in Z_i$) to be in the center of $\evv$.
\end{ex}

In the next section we will show that $Z\subseteq\{z^*:\ z\in Z_0\}$.

\section{Bicommutants of linear operators}

From now on  $R$ will  denote the algebra $\f[t]$ of polynomials  with coefficients in 
a field $\f$, and $K=\f(t)$ will be the field of all rational functions. A vector space $V$ over $\f$ will be
regarded as an $R$-module via $t\mapsto a$, where $a\in\lv$ will be fixed. Thus, for $p\in R$ and $\xi\in V$,
$p\xi$ means $p(a)\xi$. By a {\em prime} in $R$ we mean an irreducible polynomial with the leading coefficient 1.

\subsection{Preliminaries}

We would like to describe the center $Z$ of $\evv$ (which is just $\bcom{(a^*)}$). 

If $V$ is torsion-free, then $V^*$ is divisible, hence by \cite[Theorem 3.48]{La} $V^*$ is a direct sum of 
indecomposable injective
modules $W_k$. By \cite[Example 3.63]{La} every such $W_k$ is isomorphic to either  
$K=\f(t)$ or to one of the primary summands $R(p^{\infty})$ of the torsion module $K/R$, where $p\in R$ is prime. 

We note that 
$R(p^{\infty})$ is equal to the union of the increasing sequence of cyclic
submodules $C_n=R(p^{-n}+R)\cong R/(p^n)$ ($n=1,2,\ldots$), which are invariant under all endomorphisms of 
$R(p^{\infty})$. Since the endomorphism algebra of the cyclic module $R/(p^n)$ is isomorphic to $R/(p^n)$ (namely,
each endomorphism is determined by the image of the generator $1+(p^n)$), it follows that the 
endomorphism algebra of $R(p^{\infty})$  can be identified with 
$$\lim_{\leftarrow}{\rm L}_R(C_n)=\lim_{\leftarrow}R/(p^n)=:\hat{R}_p.$$
This algebra is analogous to the ring of $p$-adic integers \cite[p. 54]{AF}; its elements can be 
regarded as formal power series of the form 
\begin{equation}\label{50}f=\sum_{j=0}^{\infty}f_jp^j,\end{equation}
where $f_j\in R$ are of degree less than the degree of $p$.  
Each such series acts as an endomorphism of $R(p^{\infty})$ by multiplication: 
$f\cdot p^{-k}=\sum_{j=0}^{k-1}f_jp^{j-k}$ modulo $R$ ($k=1,2,\ldots$). Observe that on each cyclic submodule
$C_n$ this multiplication by $f$ has the same effect as the multiplication by the polynomial  
$\sum_{j=0}^{n-1}f_jp^j\in R$. Further, in this way $R/(p^n)\cong C_n$ becomes a module over 
$\hat{R}_p$, and an $R$-module homomorphism of such modules is automatically an $\hat{R}_p$-module homomorphism. We 
shall also need to know that there is a monomorphism of 
rings $\iota: R_p\to \hat{R}_p$, where $R_p:=\{u/v:\, u,v\in R,\ v\ \mbox{prime to}\ p\}$. By definition $\iota(u/v)$
is the element in $\lim_{\leftarrow}R/(p^n)$ with the component in $R/(p^n)$ equal to the class of $u/v$ in $R/(p^n)$. 
(Since $v$ is not divisible by $p$, the class $(v)_n$ is invertible in $R/(p^n)$,
hence $(u/v)_n:=(u)_n(v)_n^{-1}$ is meaningful.) We will thus regard $R_p$ as a subring in $\hat{R}_p$.

\subsection{Torsion-free modules}
\begin{lemma}Let $V$ be torsion free. Then in the decomposition of $V^*$ into a direct sum of indecomposable injective 
modules at least one summand must be $K$, hence the decomposition takes the form 
\begin{equation}\label{140}V^*=K^{(n_0)}\oplus(\oplus_{p\in P_V}R(p^{\infty})^{(n_p)}),\end{equation}
where $P_V$ is the set of all primes $p\in R$ such that ${\rm ann}_{V^*}(p)\ne0$ and $n_0\ne0$, $n_p\ne0$ are cardinal numbers.
\end{lemma}
\begin{proof} Otherwise $V^*$ would be 
a torsion module, hence such would also be every quotient of $V^*$. But, since $V$ is torsion-free, it contains a
copy of $R$, hence $R^*$ is a quotient of $V^*$. However, by Example \ref{ex1}  
$R^*$ is not torsion.
\end{proof}

With respect to the decomposition (\ref{140}) each endomorphism of $V^*$
is represented by a matrix of homomorphisms between various summands. An element of the center
$Z$ of $\evv$ must commute in particular with the projections onto the summands, so is represented by a diagonal matrix 
$b$ with the diagonal entries in $K$ or in $\hat{R}_p$. 
If $r_0\in K$ and
$f\in\hat{R}_p$ are such entries, then 
\begin{equation}\label{70}fg(r)=g(r_0r)\ \ (r\in K)\end{equation}
for each $R$-module homomorphism $g:K\to R(p^{\infty})$, since $g$ can be extended to the endomorphism
of $V^*$ by $0$ on other positions in the corresponding matrix, and $b$ must commute with this extension. 

\begin{lemma}For a prime $p\in R$ the  presence of the direct summand $R(p^{\infty})$ in the
decomposition (\ref{140}) of $V^*$ implies the equality 
$r_0=f$ in $\hat{R}_p$. (More precisely, $f$ must be equal to the image of $r_0$ in $\hat{R}_p$ under the  ring monomorphism
$R_p\to \hat{R}_p$.) This equality is also sufficient for (\ref{70}).
\end{lemma}
\begin{proof}Let $g$ be 
the composition of the quotient map
$K\to K/R$ followed by the projection of $K/R$ onto its $p$-primary component. The effect of $g$ on any rational function
$r$ can be seen by expanding $r$ into partial fractions with denominators powers of primes $q\in R$:
$g$ annihilates all the terms with denominators of the form $q^n$, $q\ne p$ ($n=1,2,\ldots$),  and leaves the terms with denominators
of the form $p^n$ unchanged, so the kernel of $g$ is $K\cap\hat{R}_p=R_p$.

Set $r=1$ in (\ref{70}). Since $g(1)=0$, it follows that $g(r_0)=0$, which  means 
that $r_0\in K\cap\hat{R}_p$. If we  show that 
\begin{equation}\label{71}g(r_0r)=r_0g(r)\ \ (r\in K),\end{equation}
then from (\ref{70}) and (\ref{71}) we will have $fg(r)=r_0g(r)$ for all $r$, hence $f=r_0$ in $\hat{R}_p$ since $g$ is surjective. 
Let $r_0=u/v$, where $u,v\in R$ are relatively prime; then $r_0$ is in
$\hat{R}_p$ if and only if $v$ is not divisible by $p$, and then $v$ is invertible in $\hat{R}_p$. For any $R$-module
map $g:K\to R(p^{\infty})$ we have $vg(r_0r)=g(vr_0r)=g(ur)=ug(r)$, hence $g(r_0r)=(u/v)g(r)=r_0g(r)$.

Conversely, if $r_0=f$ in $\hat{R}_p$, then since (\ref{71}) holds for all module homomorphisms $g:K\to R(p^{\infty})$,
(\ref{70}) also holds.
\end{proof}

Observe that the summand $R(p^{\infty})$ is present in the decomposition of $V^*$ if and only if $\ker p(a^*)\ne0$, 
which is equivalent to $p(a)V\ne V$. 
Note also that, there are no nonzero $R$-module homomorphisms from $R(p^{\infty})$ to $K$ since $R(p^{\infty})$ is torsion,
while $K$ is torsion-free. Thus, we may summarize the  discussion so far in the following proposition.

\begin{pr}\label{pr4}
If $V$ is a torsion-free module over $R=\f[t]$, then the center $Z$ of $\evv$ is isomorphic to 
the algebra $A:=\cap_{p\in P_V}(K\cap\hat{R}_p)$, where $K=\f[t]$ and $P_V$ is the set of all primes $p\in R$ such that
$pV\ne V$.  In other words, if an operator $a\in\lv$ is such that $p(a)$ is injective
for all $p\in R$, then the bicommutant $\bcom{(a^*)}$ is isomorphic to the subalgebra $A$ of $K$ consisting of all 
$r=u/v$ ($u,v\in R$ relatively prime) such that $v(a)$ is invertible. 
\end{pr}

\subsection{Torsion modules} The center $Z_0$ of $\ev$ for a torsion module $V$ is already known  \cite{Ka}. 
We will now describe the result in a form needed later and show that $x\mapsto x^*$ is an isomorphism of $Z_0$ onto
the center $Z$ of $\evv$.
Let $P$ be the set of all primes in $R$ 
and for each $p\in P$ let $V_p=\{\xi\in V:\, 
p^k(a)\xi=0\ 
{\rm for\ some}\ k\in\bn\}$, the {\em $p$-primary part} of $V$. It is well-known \cite{Ka} that 
\begin{equation}\label{80}V=\oplus_{p\in Q_V}V_p,\end{equation}
where $Q_V$ is the set of all $p\in P$ such that $V_p\ne0$. 

Let us first consider the case when there is only one summand in this decomposition, that is $V=V_p$ for some $p\in P$.
In this case $V$ is equal to the union of the increasing sequence of submodules $U_n:=\ker{p^n(a)}$. 
If $V=U_n$ for
some $n$ then $a$ is algebraic and therefore each $b\in\bcom{(a)}$ is of the form $b=f(a)$ for some polynomial $p$
by \cite[Exercise 90, p. 72]{Ka}. In this case $Z_0$ is isomorphic to $R/(p^n)$.
In the remaining case, when $U_n\ne V$ for all $n$, $Z_0$ is isomorphic to the ring $\hat{R}_p$ of all power series
of the form (\ref{50}). 

\begin{de}For each series $f$ of the form (\ref{50}) define $f(a)$ 
as follows.  Each $\xi\in V$ is in some $U_m$. Choose a polynomial $F_m\in R$ such that the coset of $F_m$ in
$R/(p^m)$ is equal to the image of $f$ under the map $\hat{R}_p=\lim_{\leftarrow}R/(p^n)\to R/(p^m)$ and then let
$$f(a)\xi:=F_m(a)\xi.$$
(We may simply take $F_m=\sum_{j=0}^{m-1}f_jp^j$, where the $f_j\in R$ are as in (\ref{50}).)
\end{de}

If $G_m$ is another such polynomial, then $G_m-F_m$ is divisible by $p^m$, hence $G_m(a)\xi=F_m(a)\xi$. Further,

if we enlarge $m$ to $m+1$, then $F_{m+1}$ and $F_m$ have the same coset in $R/(p^m)$, hence again $F_{m+1}(a)\xi=
F_m(a)\xi$. Thus $f(a)$ is well defined.
Since each $U_m$ is invariant under $\com{(a)}$, clearly  $f(a)\in\bcom{(a)}$.

It follows from known results (see \cite[Theorem 29]{Ka} or \cite[Proof of Theorem 19.7]{KMT}) that every  
$b\in\bcom{(a)}$ is of the form $f(a)$ for an $f$ as in (\ref{50}). Since all the proofs of this which we have found in
the literature require a more extensive knowledge of the structure theory of torsion modules than it is absolutely
necessary, we will sketch now a more direct argument.  

Clearly $\ev=\lim_{\leftarrow}{\rm L}_R(U_n)$, since each sequence of endomorphisms $g_n\in {\rm L}_R(U_n)$
satisfying $g_{n+1}|U_n=g_n$ defines an endomorphism $g\in\ev$. If we prove that every endomorphism
$b_n$ of the $R$-module $U_n$ extends to an endomorphism $b_{n+1}$ of $U_{n+1}$, then it follows easily that 
the center of $\ev$ is the inverse limit of the centers of ${\rm L}_R(U_n)$. These centers are isomorphic to $R/(p^nR)$ since $a|U_n$
is algebraic with the minimal polynomial $p^n$. (This is \cite[Exercise 88]{Ka}; perhaps the simplest proof is
by using the fact that $R/(p^n)$ is a self-injective ring by \cite[Corollary 3.13]{La}.) Now $U_{n+1}$ is a torsion module
with $p^{n+1}U_{n+1}=0$, hence by Pr\" ufer's theorem (which can again be seen as a consequence of self-injectivity
by \cite[Exercise 18, p. 115]{La}) $U_{n+1}$ is a direct sum of cyclic modules. The endomorphism rings of such modules
can be described quite explicitly. For example, if $U_{n+1}$ is cyclic, then so is $U_n$ and all endomorphisms of $U_{n+1}$
and $U_n$ are polynomials in $a$. In general, endomorphism of $U_{n+1}$ are suitable
matrices of homomorphisms between the cyclic summands of $U_{n+1}$. By a closer examination of this structure it can be verified that each endomorphism of
$U_n$ extends to $U_{n+1}$.

For a torsion module $V$ with the primary decomposition (\ref{80}) the center of $\ev$ is just the product
of the centers of ${\rm L}_R(V_p)$. Further, $V^*=\prod_{p\in Q_V}V_p^*$ and, since each central endomorphism of $V^*$ commutes
with the projections onto the factors $V_p^*$, the center $Z$ of $\evv$ is contained in the product of the centers $Z_p$ of 
${\rm L}(V_p^*)_R$. 

\begin{lemma}$Z=\prod_{p\in Q_V}Z_p$.
\end{lemma}

\begin{proof}To prove the remaining inclusion $\prod_pZ_p\subseteq Z$, it suffices to show that for each $q\in Q_V$ 
there are no non-zero homomorphisms
$$\phi:W_q^*=\prod_{p\ne q,\, p\in Q_V}V_p^*\to V_q^*,\ \ \mbox{where}\ W_q:=\oplus_{p\ne q,\,p\in Q_V}V_p,$$
for then each endomorphism of $V$ decomposes into the cartesian product of endomorphisms of the factors $V_p$. 
That $\phi=0$ follows from the following two observations. (1) The multiplication by $q$ acts as an invertible operator on each 
$V_p$ for $p\ne q$,
hence also invertible on $W_q$ and on $W_q^*$; in particular elements of $W_q^*$ are divisible by $q$. (2) On the other
hand no nonzero element of $V_q^*$ can be divisible by all powers $q^n$ of $q$ since $V_q$ is torsion. 
(Indeed, let $\rho\in V_q^*$ be such that for each $n\in\bn$ there exists an $\omega_n\in V_q^*$ with $\rho=\omega_nq^n$.
Then, given $x\in V_q$,
we choose for $n$ the order of $x$ and conclude that $\rho(x)=\omega_n(q^nx)=\omega_n(0)=0$.)
\end{proof}

This reduces the study of the center to the case of just one primary summand, say $V=V_p$. Then as above $V$ is the union of
submodules $U_n=\ker p^n(a)$, hence $V^*=\lim_{\leftarrow}U_n^*$. Since $U_n^*=V^*/(\ker p^n(a))^{\perp}$ and $(\ker p^n(a))^{\perp}=
p^n(a^*)(V^*)$, we may write  $U_n^*=V^*/(p^nV^*)$, hence
$$V^*=\lim_{\leftarrow}V^*/(p^nV^*).$$
Here the connecting maps in the inverse system are the natural quotient maps $$\sigma_n:V^*/(p^{n+1}V^*)\to V^*/(p^nV^*).$$

Each endomorphism $\phi$ of $V^*$ maps $p^nV^*$ into itself, hence it defines and endomorphism $\phi_n$ of
$V^*/(p^nV^*)$ for every $n$. 

\begin{lemma}\label{le14}If  $\phi_n=0$ for all $n$ then $\phi=0$. Hence $\phi$ is just a sequence of endomorphisms
$\phi_n\in{\rm L}(V^*/(V^*p^n))_R$ which are compatible in the sense that $\sigma_n\phi_{n+1}=\phi_n\sigma_n$.
\end{lemma}

\begin{proof}The identity $\phi_n=0$ means that $\phi(V^*)\subseteq p^nV^*$. Thus, if $\phi_n=0$ for all $n$, then
for each $\rho\in V^*$ the element $\phi(\rho)$ is divisible by all $p^n$. But since $V$ is $p$-primary, a similar simple
argument as in the previous lemma shows that $V^*$ contains no nonzero elements divisible by all $p^n$.
\end{proof}

\begin{lemma}\label{le15}Each endomomorphism $\psi_n\in{\rm L}(U_n^*)_R$ can be lifted to an endomorphism 
$\psi_{n+1}\in{\rm L}(U_{n+1}^*)_R$.
(That is, $\psi_n\sigma_n=\sigma_n\psi_{n+1}$.)
\end{lemma}

\begin{proof}The module $U_n^*=V^*/(p^nV^*)$ is $p^n$-torsion, hence  a direct sum of cyclic modules by \cite[pp. 17, 36]{Ka}, 
say $U_n^*=\oplus_{j\in J}R(\xi_j+p^nV^*)$, 
where $\xi_j\in V^*$
and the module $R(\xi_j+p^nV^*)$ is isomorphic to $R/(p^{n_j})$ for some $n_j\in\bn$. 
Then for each $j$ the cyclic submodule $R(\xi_j+p^{n+1}V^*)$ of
$V^*/(p^{n+1}V^*)$ is isomorphic to either $R/(p^{n_j})$ or to $R/(p^{n_j+1})$. This follows from 
$p^{n_j+1}\xi_j\in p^{n+1}V^*$ (since $p^{n_j}\xi_j\in p^nV^*)$
and  $p^{n_j-1}\xi\notin p^{n+1}V^*$ (otherwise we would have $p^{n_j-1}\xi\in p^{n+1}V^*\subseteq p^nV^*$, 
but then $R(\xi_j+p^nV^*)$ would not
be isomorphic to $R/(p^{n_j})$). Further, the sum $\sum_jR(\xi_j+p^{n+1}V^*)$ in $V^*/(p^{n+1}V^*)$ is direct 
(since its image in $V^*/(p^nV^*)$ is direct)
and equal to $U_{n+1}^*=V^*/(p^{n+1}V^*)$. To show this, note that $V^*=\sum_jR\xi_j+p^nV^*$ (since $U_n^*=\oplus_jR(\xi_j+p^nV^*)$), hence by iteration
$V^*=\sum_jR\xi_j+p^n(\sum_jR\xi_j+p^nV^*)\subseteq\sum_jR\xi_j+p^{2n}V^*\subseteq\sum_jR\xi_j+p^{n+1}V^*$. 
Now observe that each homomorphism
$\phi:R\xi_1+p^nV^*\to R\xi_2+p^nV^*$ is a multiplication by a polynomial $q$, such that $p^{n_2-n_1}|q$ if $n_2>n_1$, 
and can be lifted to a homomorphism
$\psi:R\xi_1+p^{n+1}V^*\to R\xi_2+p^{n+1}V^*$. Indeed, we can take for $\psi$ the multiplication by $q$, except in the case when $R\xi_1+p^{n+1}V^*$ is
isomorphic to $R/(p^{n_1+1})$ and $R\xi_2+p^{n+1}V^*$ is isomorphic to $R/(p^{n_2})$ since in this case $p^{n_2-n_1+1}$ 
does not necessarily divide $q$.
But in this case $R\xi_2+p^nV^*$ and $R\xi_2+p^{n+1}V^*$ are both isomorphic to $R/(p^{n_2})$, hence to each other by the homomorphism
$\tau$ sending $\xi_2+p^nV^*$ to $\xi_2+p^{n+1}V^*$. Then  we can take for $\psi$ the composition of the
natural map $R\xi_1+p^{n+1}V^*\to R\xi_1+p^nV^*$ followed by $\phi$ followed by $\tau$. Since an endomorphism $\psi_n$ of
$U_n^*$ is just a suitable matrix of endomorphisms between its cyclic summands, it follows that $\psi_n$ can be lifted 
to an endomorphism of $U_{n+1}^*$.
\end{proof}

Since an endomorphism $\phi$ of $V^*$  is just a sequence $(\phi_n)$ of compatible endomorphisms 
$\phi_n:U_n^*\to U_n^*$ by Lemma \ref{le14} and all such endomorphisms can be lifted by Lemma \ref{le15}, we deduce 
that the center $Z$ of $\evv$ is just the
inverse limit of the centers of ${\rm L}(U_n^*)_R$. The center of  ${\rm L}(U_n^*)_R$ is isomorphic to $R/(p^{m_n})$, where 
$m_n$ is the order of $U_n^*$ (the smallest natural number such that $p^{m_n}U_n^*=0$). Since $U_m^*$ has the
same order as $U_m$, we may state the following proposition.

\begin{pr}\label{pr11}If $V$ is a torsion $R$-module, then the map $z\mapsto z^*$ is an isomorphism from the center $Z_0$ of $\ev$ to the center $Z$ of
$\evv$. More precisely, if $V$ is $p$-primary for a prime $p\in R=\f[t]$, then $Z$ and $Z_0$ are both isomorphic to either the algebra
$R/(p^n)$ (where $n\in\bn$ is the minimal such that $p^nV=0$ if such a $n$ exists) or to $\hat{R}_p$ (if $V$ is not of bounded torsion).
\end{pr}

\subsection{Mixed modules}
To describe the center $Z$ of $\evv$ for a general $R$-module $V$, let $T$ be the torsion submodule of $V$ and $W=V/T$. 
Assume that $W\ne0$. Since $W$ is torsion-free, $W^*$ is divisible, hence $V^*\cong T^*\oplus W^*$. Each central endomorphism $b$ of $V^*$ 
commutes in particular  with the  projections
of $V^*$ onto the two summands $T^*$ and $W^*$, therefore it must be of the form 
$$b=f\oplus r$$ for some 
endomorphisms $f$ and $r$ of $T^*$ and $W^*$ (respectively). By Proposition \ref{pr4}  $r$ is essentially a 
rational function.
Let $T=\oplus_{q\in Q_T}T_q$ be the decomposition of $T$  into its primary summands, and let $f_q\in{\rm Z}({\rm L}(T_q)^*)$
be the components of $f$. (So $f_q\in\hat{R}_q$ or $f_q\in R/(q^n)$ by Proposition \ref{pr11}.) Decompose $W^*$ into the direct sum of
indecomposable injective submodules of the form $K$ and $R(p^{\infty})$, let $g_p\in {\rm L}(T^*,R(p^{\infty}))_R$ and $g_0\in{\rm L}(T^*,K)_R$ be 
arbitrary, 
and denote $g_{p,q}=g_p|T_q^*$.
Then (since $b\in Z$)
\begin{equation}\label{90}rg_{p,q}=g_{p,q}f_q\ \ \mbox{for all}\ g_{p,q}\in{\rm L}(T_q^*,R(p^{\infty}))_R\ \mbox{or}
\ g_{0,q}\in{\rm L}(T_q^*,K)_R.\end{equation}

We now study the question:

\smallskip
{\em For which $q\in Q_T$ is it possible that $f_q\ne r$ in $\hat{R}_q$ (or in $R/(q^{n(q)})$) 
in spite of the condition (\ref{90})?} 

\begin{lemma}\label{le16} Let $q\in Q_T$, $r\in K$ and $\phi\in\hat{R}_q$. If $T_q$ contains a nonzero divisible submodule
or, if some nonzero quotient module of $T_q$ is divisible, then the condition
$$r\psi=\psi\phi\ \ \mbox{for all}\ \psi\in{\rm L}(T_q^*,K)_R$$ 
implies that $r=\phi$ in $\hat{R}_q$.

The same conclusion also holds if $T_q$ is a direct sum of cyclic modules, $T_q=\oplus_{j\in J} R/(q^{n(j)})$, such
that the set $N_q:=\{n(j):\, j\in J\}$ is not bounded. 
\end{lemma}

\begin{proof} If $T_q$ contains a nonzero divisible submodule, then it contains an indecomposable, necessarily $q$-torsion,
such module, hence $R(q^{\infty})$. Now $R(q^{\infty})=\lim_{\rightarrow}R/(q^n)$, where the maps in the direct system are 
the injections $\mu_n:R/(q^n)\to R/(q^{n+1})$ induced by the multiplication by $q$, hence we have the inverse system
$\mu_n^*:(R/(q^{n+1}))^*\to (R/(q^n))^*$ with $\mu_n^*$ surjective. 
For a fixed $n$ let $\omega_n\in(R/(q^n))^*$ be such that
$\omega_n(q^{n-1}(1+(q^n))=1$. Then $\omega_n$ is a cyclic vector in $(R/(q^n))^*$ (since $\omega_nq^{n-1}\ne0$), so it defines an isomorphism  $\theta_n$ from $(R(q^n))^*$ to
$R/(q^n)$ such that $\theta_n(\omega_n)=1+(q^n)$. Further, if we define $\omega_{n+1}\in(R/(q^{n+1}))^*$ to be an extension of $\omega_n\mu_n^{-1}$,
then we will have that $\omega_{n+1}(q^n+(q^{n+1}))=\omega_n(q^{n-1}+(q^n))=1$. If $\sigma_n:R/(q^{n+1})\to R/(q^n)$ is the map induced by the reduction 
of polynomials modulo $(q^n)$ (note that such are the maps in the inverse system defining $\lim_{\leftarrow}R/(q^n)=\hat{R}_q$), then $\sigma_n\theta_{n+1}=
\theta_n\mu_n^*$. Thus we can inductively define a compatible sequence $(\theta_n)$ of $R$-module isomorphisms $\theta_n:(R/(q^n))^*\to R/(q^n)$,
which then induces an isomorphism $\theta$ from
$(R(q^{\infty}))^*=\lim_{\leftarrow}(R/(q^n))^*$ onto $\lim_{\leftarrow}R/(q^n)=\hat{R}_q$. Observe that $\theta$ is then
also a homomorphism of $\hat{R}_q$ modules. Let $\pi:T_q^*\to \hat{R}_q$ 
be the composition $\pi=\theta\pi_0$, where $\pi_0:T_q^*\to (R(q^{\infty}))^*$ is the quotient map
(the adjoint of the inclusion 
$R(q^{\infty})\to T_q$). Since $K$ is injective, the inclusion $K\cap\hat{R}_q\to K$ can be
extended to an $R$-module homomorphism $h:\hat{R}_q\to K$. Let $\psi=h\pi$. Then, from the condition $r\psi=\psi\phi$
we compute for every $\xi\in T_q^*$, since $\pi$ is an $\hat{R}_q$-module homomorphism,
$$rh\pi(\xi)=r\psi(\xi)=\psi(\phi\xi)=h\pi(\phi\xi)=h(\phi\pi(\xi)).$$
Since $\pi$ is surjective, this implies that
\begin{equation}\label{36}rh(\eta)=h(\phi\eta)\ \ \mbox{for all}\ \eta\in\hat{R}_q.\end{equation}
If we take $\eta=1$, we get (since $h|K\cap\hat{R}_q$ acts as the identity) 
\begin{equation}\label{64}r=h(\phi).\end{equation}
If we can show that $\phi\in K\cap\hat{R}_q$, then $h(\phi)=\phi$ and we will have $r=\phi$ as claimed. Suppose that
$\phi\notin K\cap{R}_q$, that is $\phi\notin K$. Thus $z\phi\ne s$ for all nonzero $z\in R$ and $s\in K$,
hence each element of $(K\cap\hat{R}_q)+R\phi$ can be uniquely expressed as $s+z\phi$ ($s\in K\cap\hat{R}_q$, $z\in R$). But then, for any
$r_0\in K$, we can first extend
the inclusion $K\cap\hat{R}_q\to K$  to the $R$-module map $h_0:(K\cap\hat{R}_q)+R\phi\to K$ by
$h_0(s+z\phi)=s+zr_0$, and then further extend $h_0$ to an $R$-module map $h:\hat{R}_q\to K$. For such an $h$ we have $h(\phi)=r_0$, which
contradicts (\ref{64}) if we choose $r_0\ne r$. 

If a nonzero quotient $D$ of $T_q$ is divisible, then $R(q^{\infty})$ must be a quotient of $T_q$ (since $R(q^{\infty})$
is a direct summand, hence also a quotient, of $D$), hence $(R(q^{\infty}))^*$
is a submodule of $T_q^*$. With $\theta$ and $h$ as in the previous paragraph, let now $\psi:T_q^*\to K$ be an
$R$-module homomorphic extension of $h\theta$. Now the equality $r\psi(\xi)=\psi(\phi\xi)$ holds in particular
for all $\xi\in(R(q^{\infty}))^*$ and, since $\theta$ is an isomorphism of $\hat{R}_q$-modules, this implies
that $rh(\eta)=h(\phi\eta)$ for all $\eta\in\hat{R}_q$. The argument from the previous paragraph (following (\ref{36})) shows now that
$r=\phi$ in $\hat{R}_q$.

If $T_q$ is a direct sum of cyclic modules, as in the second part of the lemma, such that $N_q$ is not bounded,
choose a sequence $(j_k)_{k\in\bn}\subseteq J$ such that $n(j_k)\geq k$ for all $k\in\bn$. 
Then we have natural $\hat{R}_q$-module monomorphisms $\iota_k:R/(q^k)\to R/(q^{n(j_k)})$ (multiplications by suitable powers of $q$), 
which induce an embedding
$\iota:\prod_{k\in\bn}R/(q^k)\to\prod_{k\in\bn}R/(q^{n(j_k)})$. There is also a natural embedding $\kappa:\hat{R}_q\to
\prod_{k\in\bn}R/(q^k)$ of $\hat{R}_q$-modules, given by 
$$\kappa(\sum_{i=0}^{\infty}c_iq^i)=([c_0],[c_0+c_1q],\ldots,[\sum_{i=0}^nc_iq^i],\ldots)\ \ (c_i\in R,\ {\rm degree}\, 
c_i<{\rm degree}\, q).$$
Finally, since $(j_k)_{k\in\bn}$ is a subset of $J$, we have the obvious embedding 
$$\tau:\prod_{k\in\bn}R/(q^{n(j_k)})\rightarrow \prod_{j\in J}R/(q^{n(j)}).$$ Let $\sigma:\hat{R}_q\to T_q^*$ be the 
composition $\sigma=\theta^{-1}\tau\iota\kappa$, where the isomorphism $$\theta:T_q^*=\prod_{j\in J}(R/(q^{n(j)}))^*\to 
\prod_{j\in J}R/(q^{n(j)})$$ is defined by some compatible family of isomorphisms $R/(q^{n(j)})^*\to
R/(q^{n(j)})$ of $\hat{R}_q$-modules (as in the first paragraph of this proof). Since all these are $\hat{R}_q$-module maps, we may regard $\hat{R}_q$ 
as a submodule in $T_q^*$. Now the proof can be completed by the argument from the first paragraph of this proof (following
(\ref{36})) by considering for $\psi$ an extension to $T_q^*$ of the map $h:\hat{R}_q\to K$, where $h$ is an appropriate extension
of the inclusion $\hat{R}_q\cap K\to K$.
\end{proof}

If $T_q$ does not have any nonzero divisible submodule, then there exists in $T_q$ a submodule $B_q$ such that
$B_q$ is a direct sum of cyclic modules and $T_q/B_q$ is divisible. (This follows from Kulikov's theorem \cite [4.3.4]{Rob} for $\bz$-modules,
the proof for $\f[t]$-modules is essentially the same.) So, unless $T_q=B_q$, $T_q$ has a nonzero divisible
quotient and therefore by Lemma \ref{le16} the condition (\ref{90}) implies that $f_q=r$. In the remaining case,  
$T_q=B_q$, $T_q$ is a direct sum of cyclic modules as in the second part of Lemma \ref{le16}, hence, if the set
$N_q$ is not bounded (\ref{90}) again implies that $f_q=r$. 
On the other hand, if the set 
$N_q$ is bounded, say by the least upper bound $n(q)$, 
then $q^{n(q)}T_q=0$ implies that $T_q^*$ is a torsion module, hence there can be no nonzero module maps from $T_q^*$ to 
$K$. But there are nonzero such maps from $T_q^*$ to $R(p^{\infty})$ if $p=q$ for some $p\in P_W$, where $P_W$ is the 
set of all primes $p\in R$ such that $W^*$ contains $R(p^{\infty})$. Namely, in this case  $R(q^{\infty})$
and $T_q^*$ each contains its own  copy of $R/(q^{n(q)})$ and if we choose $g_{q,q}$ so that it identifies these
two copies isomorphically, then we see easily that (\ref{90}) implies that $f_q=r$ modulo
$(q^{n(q)})$. This proves:

\begin{lemma} Let $q\in Q_T$ and $r\in K$ be as above (so that (\ref{90}) holds). If $f_q\ne r$, then
$q$ is necessarily in the subset $S$ of $Q_T\setminus P_W$ consisting of those $q$ for which $T_q$ is a direct sum 
of cyclic modules of bounded orders; let $n_q$ be the order of $T_q$.
\end{lemma}

Finally, let 
$S_b=\{q\in  S:\, f_q\ne r\ \mbox{in}\ R/(q^{n(q)})\}$. 

\begin{lemma}  The set $S_b$ is  finite.
\end{lemma}
\begin{proof}If not, then there exists a sequence $S_1=\{q_0,q_1,\ldots\}\subseteq S_b$
such that $q_j\ne q_i$ if $j\ne i$. Consider the module
$$(\bigoplus_{j\in\bn} R/(q_j^{n_j}))^*\cong\prod_{j\in \bn}(R/(q_j^{n_j}))^*=:U,$$
where $n_j=n(q_j)$. Note that $U$ is a direct summand in $T^*$ since $T$ is the direct sum of modules $T_q$ and each $R/(q_j^{n_j})$
is a direct summand in $T_q$. Let $r=u/v$ with $u,v$ relatively prime polynomials. Let 
$\omega\in U$ be defined by $\omega=(\omega_j)_{j\in\bn}$, where $\omega_j\in (R/(q_j^{n_j}))^*$ is such that
$q_j^{n_j-1}\omega_j\ne0$, and let $\rho=(\rho_j)_{j\in \bn}$ be $\rho=v\omega$. For any $g_0\in {\rm L}(T^*,K)_R$ we have
that $rg_0(\rho)=g_0(f\rho)$ (since $(f\oplus r)\in Z$), hence $ug_0(\rho)=vg_0(f\rho)$, consequently $vg_0(u\omega)
=ug_0(v\omega)=ug_0(\rho)=vg_0(f\rho)=vg_0(fv\omega)$,
thus
$$g_0(u\omega-vf\omega)=0.$$
Since this holds for all $g_0$ and $K$ is injective, this implies that the element $u\omega-vf\omega$ is torsion,
say $h(u-vf)\omega=0$ for a polynomial $h\in R$.  
Since the components of this element are $h(u-vf_{q_j})\omega_j$, it follows from the definition of $\omega$ that 
$q_j^{n_j}$ 
divides $h(u-vf_{q_j})$. Since $h$ has only finitely many divisors, it follows that for all sufficiently large
$j$, $u-vf_{q_j}$ must be divisible by $q_j^{n_j}$, say $u-vf_{q_j}=s_jq_j^{n_j}$ for a polynomial $s_j$. Then
\begin{equation}\label{100}\frac{u}{v}=f_{q_j}+\frac{s_j}{v}q_j^{n_j}.\end{equation}
Since $v$ can be divisible only by finitely many prime factors $q_j$, $1/v$ acts as a multiplication by a polynomial on
$(R/(q_j^{n_j}))^*$ if $j$ is large enough. (If $1=m v + zq^{n_j}$ for some polynomials $m,z$, then $1/v$ acts as 
the multiplication by $m$.) Thus (\ref{100}) implies that $r$ and $f_{q_j}$ coincide as operators on  
$(R/(q_j^{n_j}))^*$ for large enough $j$ (since $q_j^{n_j}$ acts as $0$), hence also on $R/(q_j^{n_j})\cong (R/(q_j^{n_j}))^*$, but this contradicts the definition of $S_b$.
\end{proof}

Observe that the submodule $T_b:=\oplus_{q\in S_b}T_q$ of $V$ is {\em pure} in the following sense:  given 
$\eta\in T_b$ and $p\in R$, if there exists $\xi\in V$ such that $p\xi=\eta$, then there exists $\zeta\in T_b$
such that $p\zeta=\eta$. Since $S_b$ is finite, we can form  $q_b=\prod_{q\in S_b}q^{n_q}$
and clearly $q_bT_b=0$. It follows now from \cite[Theorem 7]{Ka} that $T_b$ is a direct summand in $V$,
say $V=T_b\oplus W_b$. Let $T_0=\oplus_{q\in Q_T\setminus S_b}T_q$, so that $T=T_b\oplus T_0$. Then 
$$T_b^*\oplus W_b^*=V^*=T^*\oplus W^*=T_b^*\oplus T_0^*\oplus W^*,$$
hence $W_b^*$ is naturally isomorphic to $T_0^*\oplus W^*$ (since they are both isomorphic to $V^*/T_b^*$). Since our element $b=f\oplus r$ acts on $T_0^*\oplus W^*$
as the multiplication by $r$ (by the definition of $S_b$), the same must hold for the action of $b$ on $W_b^*$ and consequently also on $W_b$.

The above discussion and lemmas prove the following theorem in the harder direction.

\begin{theorem}\label{th21}Let $T$ be the torsion submodule of a module $V$ over $R=\f[t]$, $W=V/T$ and $Z$ the center
of $\evv$. Denote by $P_W$ the  set of all primes $p\in R$ such that $pW\ne W$, by $Q_T$ the set of all primes 
$q\in R$ such that the $q$-primary part $T_q$ of $T$ is nonzero, and by $S$ the set of all $q\in Q_T\setminus P_W$ such 
that $T_q$ is a direct sum of cyclic modules of orders bounded by $n_q$ (that is, $q^{n_q}T_q=0$ for $n_q\in\bn$ and we 
choose the minimal such $n_q$). Suppose that $W\ne 0$. Then for each $b\in Z$ there exist a finite subset $S_b$ of 
$S$ and a submodule $W_b$ of $V$ such that 
$$V=T_b\oplus W_b,\ \ \mbox{where}\ \ T_b=\oplus_{q\in S_b}T_q,$$
b acts on $W_b$ as the multiplication by a rational function $r\in K$,
and $b$ acts on each summand $T_q$ of $T_b$ as the multiplication by a polynomial $f_q$. Conversely, any map $b$
on $V$ for which there exist such a decomposition of $V$ and $b$ is in center $Z$ of $\evv$ and in the center $Z_0$ of $\ev$.
\end{theorem}
\begin{proof} It only remains to prove the sufficiency of the stated conditions for $b$ to be in $Z$ and in $Z_0$. So, let
$b=(\oplus_{q\in S_b}f_q)\oplus r$ be the decomposition of $b$ with respect to the decomposition
\begin{equation}\label{102}V=(\oplus_{q\in S_b}T_q)\bigoplus W_b\end{equation}
of $V$, where $r\in K$ and $f_q\in R/(q^{n_q})$. Because of incompatible torsion there can be no nonzero module
homomorphisms between $T_{q_1}^*$ and $T_{q_2}^*$ (or between $T_{q_1}$ and $T_{q_2}$) for different $q_1,q_2$ in $S_b$. If we can show that there are no
nonzero homomorphisms from $T_q^*$ to $W_b^*$ and from $W_b^*$ to $T_q^*$ (consequently also no nonzero homomorphism
between $T_q$ and $W_b$), then each endomorphism of $V$ (and of $V^*$) will be represented
by a diagonal matrix relative to the decomposition (\ref{102}) (relative to the corresponding decomposition of $V^*$), hence clearly $b$ will be in $Z_0$ (and in $Z$).

Note that for $q\in S_b$ the multiplication by $q$ acts as an invertible operator on each primary summand $T_p$ of 
$T_0$ (hence also on the dual $T_0^*$ of the direct sum
of such summands) and also on each summand $R(p^{\infty})$ of $W^*$ (since $q\notin P_W$). Thus for $q\in S_b$ the 
multiplication by $q$ as invertible operator on $W_b^*=T_0^*\oplus W^*$
(hence also invertible on $W_b$).  Since $q^{n_q}T_q^*=0$, 
while the multiplication by $q$ acts as an invertible operator on $W_b^*$, we have that ${\rm L}(T_q^*,W_b^*)_R=0$
and also that ${\rm L}(W_b^*,T_q^*)_R=0$. (To prove the last identity, note that for each $\phi\in {\rm L}(W_b^*,T_q^*)_R$
we have $\phi(W_b^*)=\phi(q^{n_q}W_b^*)=q^{n_q}\phi(W_b^*)=0$.)
\end{proof}

\begin{co}\label{co111}The center $Z$ of $\evv$ can be regarded as a subset of the center $Z_0$ of $\ev$ (that is, each $z\in Z$ is
of the form $c^*$ for some $c\in Z_0$). Moreover, if $V$ is torsion or injective or if $V$ contains an isomorphic copy of $R$, then $Z=\{c^*:\ c\in Z_0\}$.
\end{co} 

\begin{proof}The first sentence follows directly from Theorem \ref{th21}. If $V$ is torsion or injective it follows now from Proposition \ref{pr11}
or Corollary \ref{co2} that the map $c\mapsto c^*$ is an isomorphism of $Z_0$ onto $Z$. Finally, if $V$ contains a copy of $R$ as a direct summand
it is well-known (and elementary to prove, \cite[Exercise 95]{Ka}) that $Z_0=R$, hence by Theorem \ref{th21} $Z$ and $Z_0$ must both essentially coincide
with $R$.
\end{proof}

{\bf Problem.} For which subalgebras $A\subseteq\lv$ is the center of ${\rm L}(V^*)_A$ contained in the set
$\{x^*:\, x\in Z_0\}$, where $Z_0$ is the center of ${\rm L}_A(V)$?

\section{The range inclusion for  derivations}

Throughout this section (except in the Example \ref{ex12})  we assume that $V$ is a module over $R=\f[t]$ such that the map $x\mapsto x^*$ is an isomorphism
from the center $Z_0$ of $\ev$ onto the center $Z$ of $\evv$. In other words, we assume that $a\in\lv$ is such that

\begin{equation}\label{A}
\bcom{(a^*)}=
\{c^*:\, c\in\bcom{(a)}\}.
\end{equation}

Let $\fv$ be the vector space of all finite rank linear operators on a vector space $V$. 
$\fv$ is naturally isomorphic to $V\otimes V^*$ by the isomorphism which sends $\xi\otimes\rho$ ($\xi\in V$,
$\rho\in V^*$) to the rank 1 operator $\eta\mapsto \rho(\eta)\xi$. Thus the dual space of $\fv$ can be identified
with $(V\otimes V^*)^*=\lvv$.

\begin{theorem}\label{th3}Let $d_a$ and $d_b$ be the  derivations on $\lv$ induced by operators
$a,b\in\lv$. Assume that $a$ has property (\ref{A}). Then $\ker{d_a}\subseteq\ker{d_b}$ if and only if $d_b(\fv)\subseteq d_a(\fv)$. In this case the
inclusion $\im{d_{b^*}}\subseteq\im{d_{a^*}}$ also holds.
\end{theorem}

\begin{proof}Note that $\ker{d_a}$ is just the commutant $\com{(a)}$ of $a$ and that the inclusion
$\com{(a)}\subseteq\com{(b)}$ is equivalent to $b\in\bcom{(a)}$. (Indeed, by taking the commutants we infer
from $\com{(a)}\subseteq\com{(b)}$ that $b\in\bcom{(a)}$. Conversely, $b\in\bcom{(a)}$
implies that $(a)^{\prime\prime\prime}\subseteq\com{(b)}$; but $A^{\prime\prime\prime}=\com{A}$ for any subalgebra
$A$ of $\lv$, as it is easy to deduce  from the obvious inclusion $A\subseteq\bcom{A}$.) Thus the conditions $\ker{d_a}\subseteq\ker{d_b}$ and $b\in\bcom{(a)}$ are equivalent.
By the same argument the conditions $\ker{d_{a^*}}\subseteq\ker{d_{b^*}}$ and
$b^*\in\bcom{(a^*)}$ are also equivalent. But by assumption $b\in\bcom{(a)}$ if and only if $b^*\in\bcom{(a^*)}$,
hence it follows that the two conditions $\ker{d_a}\subseteq\ker{d_b}$ and $\ker{d_{a^*}}\subseteq\ker{d_{b^*}}$ are
equivalent. It is easy to verify that $d_{a^*}=-(d_a)^*$, hence $\ker{d_{a^*}}=(d_a(\fv))^{\perp}$ and similarly for
$b$, so the last inclusion is equivalent to $d_b(\fv)\subseteq d_a(\fv)$.

If $\ker{d_a}\subseteq\ker{d_b}$, then 
\begin{align*}\im{d_{b^*}}=\im{((d_b|\fv)^*)}=(\ker{d_b|\fv})^{\perp}\subseteq\\
(\ker{d_a|\fv})^{\perp}=\im{((d_{a}|\fv)^*)}=\im{d_{a^*}}.\end{align*}
\end{proof}

Is there any connection between the two range inclusions
\begin{equation}\label{31}d_b(\fv)\subseteq d_a(\fv)\end{equation}
and
\begin{equation}\label{32}d_b(\lv)\subseteq d_a(\lv)?\end{equation}
The following example show that (\ref{32}) does not imply (\ref{31}).

\begin{ex}\label{ex10}Let $V=\f^{(\bn)}$ and $a\in\lv$ be  as in Example \ref{ex1}, so that
any $b\in\bcom{(a)}$ is a polynomial in $a$. Since $V$ is isomorphic to $R=\f[t]$ as an $R$-module, $a$ has 
property (\ref{A}).
We may represent operators in $\lv$
by $\bn\times\bn$ matrices (that have in each column only finitely many nonzero elements). A short computation shows 
that $d_a$ is surjective on $\lv$. Thus the inclusion 
(\ref{32}) holds for all $b\in\lv$, not just
for $b\in\bcom{(a)}$. (In this respect $a$  behaves quite differently from the shift operator
on the Hilbert space $\ell_2$; for the latter see \cite{W}.) 
\end{ex}

In the rest of the paper we will prove that (\ref{31}) implies (\ref{32}), which is (by the proof of Theorem
\ref{th3}) equivalent to the following theorem.

\begin{theorem}\label{th4}Let $a,b\in\lv$ and suppose that $a$ has property (\ref{A}). If $b\in\bcom{(a)}$, then the range inclusion $d_b(\lv)\subseteq d_a(\lv)$ 
holds.
\end{theorem}

We will divide the proof of this theorem into several cases.
\subsection{The case when $\bcom{(a)}$ is isomorphic to a subalgebra of $\f[t]$}
First we will consider the simple case when the center
of $\ev$ is just $R=\f[t]$ (hence any $b\in\bcom{(a)}$ is a polynomial in $a$). 
For this and a later use, it will be convenient to have the following definition. 

\begin{de}For a polynomial $p(t)=\sum_{j=0}^n\alpha_jt^j$  the derivative of the map
$a\mapsto p(a)$ at a point $a\in\lv$ is the linear map $\dot{p}_a:\lv\to\lv$
defined by
$$\dot{p}_a(x)=\sum_{j=1}^n\alpha_j\sum_{i=0}^{j-1}a^{j-i-1}xa^i\ \ \ (x\in\lv).$$
\end{de}

Clearly $\dot{p}_a$ is an $\com{(a)}$-bimodule endomorphism of $\lv$ and a simple computation shows that
\begin{equation}\label{33}d_a(\dot{p}_a(x))=d_{p(a)}(x)\end{equation}
for all $x\in\lv$. Thus $\im{d_{p(a)}}\subseteq\im{d_a}$. Moreover, the following form of Leibnitz rule can also be easily 
verified:
\begin{equation}\label{34}(pq)\dot{}_a(x)=\dot{p}_a(x)q(a)+p(a)\dot{q}_a(x)\ \ (x\in\lv,\ p,q\in R).\end{equation}
This suggest us to define $\dot{r}_a(x)$ for any rational function
$r=p/q$, such that $q(a)$ is invertible,  to be the unique operator satisfying the following two equivalent identities:
\begin{equation}\label{120}q(a)\dot{r}_a(x)+\dot{q}_a(x)r(a)=\dot{p}_a(x)=\dot{r}_a(x)q(a)+r(a)\dot{q}_a(x);\end{equation}
that is
$$\dot{r}_a(x)=q(a)^{-1}\dot{p}_a(x)-q(a)^{-1}\dot{q}_a(x)r(a)\ \ (x\in\lv).$$ 
(That so defined $\dot{r}_a(x)$ satisfies also the second equality in (\ref{120}) can be proved by a simple computation, 
using the identity $\dot{p}_a(x)q(a)+p(a)\dot{q}_a(x)=(pq)^{\cdot}_a(x)=(qp)^{\cdot}_a(x)=\dot{q}_a(x)p(a)+q(a)\dot{p}_a(x)$.)
Then it can be proved that (\ref{33}) and (\ref{34}) hold for rational functions. (To prove that 
$(rs)\dot{}_a(x)=\dot{r}_a(x)s(a)+r(a)\dot{s}_a(x)$ for two rational functions $r$ and $s$,
one shows that $\dot{r}_a(x)s(a)+r(a)\dot{s}_a(x)$ satisfies one of the two defining identities (\ref{120}) for the product $rs$
in place of $r$.)
Then (\ref{33}) implies that the range inclusion $\im{d_b}\subseteq\im{d_a}$ holds whenever the center of $\ev$ is 
a subalgebra of $K=\f(t)$, in particular if $V$ is torsion-free, by Proposition \ref{pr4} and the assumption (\ref{A}). 

\subsection{The case of locally algebraic operators} We study next the case of torsion modules, that is, we assume that $a\in \lv$ is locally algebraic. If $a$ is algebraic, 
each $b\in\bcom{(a)}$ is of the form $b=f(a)$ for a 
polynomial $f$, and $d_b=d_a\dot{f}_a$ by (\ref{33}), hence $\im{d_b}\subseteq\im{d_a}$. 
Assume now that $a$ is not algebraic but that  $V$ is  $p$-primary for some prime $p\in R$, that is,
$V=\cup_{m=1}^{\infty}U_m$, where $U_m=\ker{p(a)^m}$. Then $b=f(a)$ for a power 
series of the form (\ref{50}) (see Subsection 3.3).  For such a series $f$ we would like to define 
$\dot{f}_a:\lv\to\lv$. 

\begin{de}Given $\xi\in V$, choose $m$ so that $\xi\in U_m$, and then choose $k$ such that
all the vectors $\dot{p}_a(x)p(a)^j\xi$ ($j=0,\ldots m$) are in $U_k$ and set $n=k+m$.
Choose a polynomial $F_n\in R$ such that the coset of $F_n$ in $R/(p^n)$ is equal to the image of $f$ under the
natural map $\hat{R}_p\to R/(p^n)$ (for example $F_n=\sum_{j=0}^{n-1}f_jp^j$, where $f_j$ are as in (\ref{50})) and set
\begin{equation}\label{35}\dot{f}_a(x)\xi=(F_n)\dot{}_a(x)\xi.\end{equation}
\end{de}

We have to show that this definition is independent of the choices of $m$, $k$ and $F_n$. 
If $G_n$ is another such polynomial, then $G_n-F_n=qp^n$ for a polynomial $q\in R$. Using the product
formula (\ref{34}) and induction we have (since $p(a)^n\xi=0$) that
\begin{align*}(qp^n)\dot{}_a(x)\xi=\dot{q}_a(x)p^n(a)\xi+q(a)(p^n)\dot{}_a(x)\xi
=q(a)\sum_{j=0}^{n-1}p(a)^{n-j-1}\dot{p}_a(x)p(a)^{j}\xi=0,\end{align*}
since $n-j-1\geq k$ for all $j=0,\ldots,m-1$. Thus $(G_n)\dot{}_a(x)\xi=(F_n)\dot{}_a(x)\xi$.
Enlarging $m$ and $k$ would give us a polynomial congruent to $F_n$ module $p^n$. This shows that $\dot{f}_a$ is well 
defined. Then $\dot{f}_a$ is an $\com{(a)}$-bimodule
map on $\lv$ such that $d_a\dot{f}_a=d_{f(a)}=d_b$ (since this holds locally by the polynomial case considered above), hence $\im{d_b}\subseteq\im{d_a}$. 

In the general case of several primary summands, relative to the decomposition $V=\oplus_{p\in P} V_p$ each operator
$x\in\lv$ is represented by an infinite operator matrix $[x_{p,q}]$, where $x_{p,q}$ is an operator from $V_q$ to
$V_p$. Denoting by $a_p$ and $b_p$ the restrictions of $a$ and $b$ to $V_p$, the proof that $\im{d_b}\subseteq\im{d_a}$ 
reduces to showing that the equation 
\begin{equation}\label{40}a_px-xa_q=y\end{equation}
has a solution $x$ in ${\rm L}(V_q,V_p)$ for each $y\in{\rm L}(V_q,V_p)$ of the form $b_pz-zb_q$. If $q=p$,
this has just been proved in the previous paragraph, while if $q\ne p$ it is a consequence of Lemma \ref{le2} below.
We remark that the equation of the form $cx-xd=y$ has been studied in the analytic context 
using the notion of spectrum (see e.g. \cite{Ros} or \cite[p. 8]{RR}), but this method does not apply to the purely 
algebraic context of Lemma \ref{le2}.

\begin{lemma}\label{le2} Let $c$, $e$ be linear operators on vector spaces $V_c$ and $V_e$. If there exists
a polynomial $v$ such that $v(e)=0$ and $v(c)$ is invertible, then  for each $y\in{\rm L}(V_e,V_c)$ the equation
\begin{equation}\label{41}cx-xe=y\end{equation}
has a unique solution $x\in{\rm L}(V_e,V_c)$. 

The same conclusion holds under the assumption that  one of the spaces $V_c$, $V_e$, say $V_e$, is of the form
$V_e=\cup_{n=1}^{\infty}\ker{q(e)^n}$ for a polynomial $q$ such that $q(c)$ is invertible.
\end{lemma}

\begin{proof}
Observe that
for any polynomial $f=\sum\alpha_jt^j$, $x\in{\rm L}(V_e,V_c)$ and $y:=cx-xe$ the following identity holds: 
\begin{equation}\label{42}f(c)x-xf(e)=\sum_j\alpha_j\sum_{i=0}^{j-1}c^{j-i-1}ye^{i}=:\dot{f}_{c,e}(y)\end{equation}
Applying this to $f=v$, if $v(e)=0$ and $v(c)$ is invertible, we deduce that
\begin{equation}\label{43}x=v(c)^{-1}\dot{v}_{c,e}(y).\end{equation}
Conversely, a direct computation (using the definition (\ref{42}) of $\dot{f}_{c,e}(y)$) verifies that $x$ given by (\ref{43}) is
indeed a solution of (\ref{41}). 

In a more general situation of the second part of the lemma, we try to solve
the equation (\ref{41}) locally, that is, for each $\xi\in V_e$ we choose $n$ so that $\xi\in\ker{q(e)^n}$ and then, 
as suggested by (\ref{43}) we set 
$$x\xi:=q(c)^{-n}(q^n)\dot{}_{c,e}(y)\xi,$$
To show that $x\xi$ is independent
of the choice of $n$, note (by a straightforward computation) for any two polynomials $u,v$ the equality
$$(uv)\dot{}_{c,e}(y)=\dot{u}_{c,e}(y)v(e)+u(c)\dot{v}_{c,e}(y).$$
Applying this to $u=q$ and $v=q^n$, we obtain (since $q(e)^n\xi=0$) that
$$q(c)^{-n-1}(q^{n+1})\dot{}_{c,e}(y)\xi=q(c)^{-n}(q^{n})\dot{}_{c,e}(y)\xi.$$
So $x$ is a well-defined map and it follows that $x$
is linear and satisfies the equation (\ref{41}).
\end{proof}

\subsection{The general non-torsion case}For a general non-torsion module $V$ we know from Theorem \ref{th21} and our assumption
(\ref{A}) that for 
every $b$ in $\bcom{(a)}$ ($=Z_0$) $V$ decomposes into a direct sum $V=(\oplus_{q\in S_b}T_q)\oplus W_b$ so that $b$ acts on $W_b$ as the 
multiplication by a rational function $r$ (that is, $b|W_b=r(a|W_b)$) and $b$ acts on each $T_q$ as the multiplication by a 
polynomial $f_q$. Moreover, $q(a)|T_p$ is invertible if $p,q\in S_b$ are different (since $p$ and $q$ are different
primes in $R$). Also $q(a)|W_b$ is invertible for $q\in S_b$, as we have seen in the proof of Theorem \ref{th21}. 
If we now represent operators in $\lv$ by matrices relative to this decomposition of
$V$, we may again use Lemma \ref{le2}, in the same way as above, to show that the
equation $ax-xa=y$ has a solution for each $y\in\im{d_b}$. 

\medskip

Now we show by an example that the conclusion of Theorem \ref{th4} is not true if we drop the assumption (\ref{A}).

\begin{ex}\label{ex12} Let $R=\f[t]$, $U$, $W$ and $V=U\oplus W$  be as in Example \ref{ex11}, so that ${\rm L}_R(U,W)=0$
and ${\rm L}_R(W,U)=0$. Let $a_1\in{\rm L}_R(U)$ and $a_2\in{\rm L}_R(W)$ be the multiplications by $t$ and set $a:=a_1\oplus a_2$,
$b:=a_1\oplus0$. Then $a,b\in\ev$ and $b\in\bcom{(a_1)}\oplus\bcom{(a_2)}=\bcom{(a)}$. Nevertheless  the range of $d_b$ is
not contained in the range of $d_a$. To show this, represent operators on $V$ as $2\times 2$ matrices relative to
the decomposition $V=U\oplus W$ and consider the entries in the position $(1,2)$: if $d_b(\lv)\subseteq d_a(\lv)$, a
simple computation shows that for each $x\in{\rm L}(W,U)$ there exists an $y\in {\rm L}(W,U)$ such that $a_1y=a_1x-xa_2$.
But then $a_1(x-y)=xa_2$, which implies that the range of $xa_2$ is contained in the range of $a_1$. This can hold
for all $x\in{\rm L}(W,U)$ only if $a_1=0$, a contradiction.
\end{ex}

\section{A counterexample among bounded operators}
An example of a Banach space $X$  and  operators $a,b\in{\rm B}(X)$ (the algebra of all bounded operators on $X$) will be 
presented
such that $b$ is in the bicommutant $\bcom{(a)}$ of $a$ in ${\rm B}(X)$, but nevertheless its adjoint $\du{b}$ is not 
in the bicommutant of $\du{a}$ in ${\rm B}(\du{X})$. Here we use the notation $\du{X}$ for the Banach space dual of $X$,
to distinguish it from the linear space dual $X^*$. Similarly, $\du{a}$ denotes the bounded adjoint operator of $a$ acting on $\du{X}$.

\begin{ex}Let $H$ be a separable infinite dimensional Hilbert space, $B$ the algebra of all bounded linear
operators on $H$, $K$ the ideal in $B$ of all compact operators, $C=B/K$ the Calkin algebra and $T$ the
ideal in $B$ of all trace class operators. It is well-known (see e.g. \cite[Section II.1]{T} or \cite{KR}) that $\du{K}=T$ and $\du{B}=T\oplus K^{\perp}$, 
where $K^{\perp}$ is the annihilator of $K$ in 
$\du{B}$, and $K^{\perp}=\du{C}$. 

Let $a_0\in K$ be an injective operator represented in some orthonormal basis of $H$ by
a diagonal matrix with the diagonal entries $\alpha_n$, with $\alpha_n\ne\alpha_m$ if $n\ne m$.
Denote by $D$ the algebra of all operators in $B$ that can be represented by diagonal matrices with respect to the
same orthonormal basis as $a_0$. Thus $D$ is the commutant and the bicommutant of $a_0$ in $B$. Finally, let
$a$ be the left multiplication by $a_0$ on $B$ (that is, $a(x)=a_0x$ for all $x\in B$). Note that 
$\du{a}$ is the right multiplication by $a_0$ on $\du{B}$ (where $\rho a_0$ is defined by $\inner{\rho a_0}{x}=
\inner{\rho}{a_0x}$, for all $\rho\in\du{B}$ and $x\in B$),  $T$ is invariant under $\du{a}$, and $\du{a}(\du{C})=0$
since $a_0$ is compact. Thus $\du{a}$ on $\du{B}=T\oplus\du{C}$ decomposes into the direct sum $\du{a}=(\du{a}|T)\oplus0$,
hence $\com{(\du{a})}$ contains all maps on $T\oplus\du{C}$ of the form $0\oplus h$, where $h$ is the right multiplication
on $\du{C}$ by any element $\dot{c}\in C$. We can choose $\dot{c}\in C$ and $d_0\in D$ so that $\dot{c}\dot{d}_0\ne
\dot{d}_0\dot{c}$, where $\dot{d}_0$ denotes the coset of $d_0$ in the Calkin algebra $C$. (Indeed, since $D$ is equal to
its own commutant in $B$, the same holds for the image of $D$ in $C$ by \cite{JP}.) Let $d:B\to B$ be the left 
multiplication by $d_0$. We will show that $d\in\bcom{(a)}$ and that nevertheless $\du{d}\notin\bcom{(\du{a})}$.

Let $f\in\com{(a)}$ (where the commutant is taken in the algebra of all bounded operators on $B$) and let $f=f_n+f_s$ be the decomposition of $f$ into the normal part $f_n$ and
the singular part $f_s$ (this means that $f_n$ is weak* continuous and $f_s(K)=0$, see \cite[Chapter 10]{KR}).
Then $f_n$ and $f_s$ are both in $\com{(a)}$. (Indeed, the equality $fa=af$ can be written as $f_na-af_n=af_s-f_sa$,
where the left side is normal and the right side is singular, so they are both $0$.) So $f_sa=af_s$, which means that
$$f_s(a_0x)=a_0f_s(x)\ \ \mbox{for all}\ x\in B.$$
Since $a_0\in K$, $f_s(a_0x)=0$, hence also $a_0f_s(x)=0$ for all $x\in B$. Since $a_0$ is injective, we deduce that
$f_s=0$, hence $f=f_n$ is normal. Let $(a_n)$ be a sequence in the algebra generated by $a_0$, converging to $d_0$ in the
weak* topology. Since $f\in\com{(a)}$, $f$ commutes with the left multiplications by all $a_n$. Then the weak* continuity of 
$f$ implies that 
$$f(d_0x)=\lim_nf(a_nx)=\lim_na_nf(x)=d_0f(x)\ \ \mbox{for all}\ x\in B.$$
Hence $d$ commutes with $f$ and therefore $d\in\bcom{(a)}$.

If $\du{d}$ were in $\bcom{(\du{a})}$, then in particular $\du{d}|\du{C}$ would commute with the right multiplication
$h$ by $\dot{c}$ on $\du{C}$, since $0\oplus h$ is in $\com{(\du{a})}$ by the second paragraph of this example. This would
imply that $\dot{d}_0\dot{c}=\dot{c}\dot{d}_0$, a contradiction with the choice of $\dot{c}$ and $\dot{d}_0$.
\end{ex}

Further, the analogy of Theorem \ref{th4} in the context of ${\rm B}(H)$
is not true even for normal operators \cite{JW}. But perhaps a more proper formulation of the problem is as follows.

\smallskip
{\bf Problem.} Suppose that operators $a,b\in{\rm B}(H)$ satisfy $\|d_b(x)\|\leq\kappa\|d_a(x)\|$ for all $x\in{\rm B}(H)$,
where $\kappa$ is a constant. Is then necessarily $d_b({\rm B}(H))\subseteq d_a({\rm B}(H))$? 

\smallskip
Perhaps the answer to the above question is negative in general, but it would be interesting to have a counterexample
and to know if the answer is positive for some large classes of operators $a$ (e.g. hyponormal).

\end{document}